\documentclass[a4paper]{amsart}
\usepackage{graphicx}
\usepackage{amsfonts}
\usepackage{amsmath}
\usepackage{amsthm}
\usepackage{amssymb}
\usepackage[all]{xy}
\usepackage{hyperref}
\usepackage{aliascnt}
\usepackage{multicol}
\usepackage{enumitem}
\usepackage{comment}
\usepackage[mathscr]{euscript}

\newtheorem{thmInt}{Theorem}[section]

\newaliascnt{prop}{thm}
\newtheorem{prop}[prop]{Proposition}
\aliascntresetthe{prop}

\newaliascnt{lem}{thm}
\newtheorem{lem}[lem]{Lemma}
\aliascntresetthe{lem}

\newaliascnt{cor}{thm}
\newtheorem{cor}[cor]{Corollary}
\aliascntresetthe{cor}

\theoremstyle{definition}

\newaliascnt{definition}{thm}
\newtheorem{definition}[definition]{Definition}
\aliascntresetthe{definition}

\newaliascnt{remark}{thm}
\newtheorem{remark}[remark]{Remark}
\aliascntresetthe{remark}

\newaliascnt{ex}{thm}
\newtheorem{ex}[ex]{Example}
\aliascntresetthe{ex}

\numberwithin{equation}{section}

\DeclareMathOperator{\im}{im} 
\DeclareMathOperator{\Proj}{Proj} 
\newcommand{\iso}{\cong}

\newcommand{\id}{{\rm id}}
\newcommand{\st}{\,|\,} 
\newcommand{\mor}[1]{\xrightarrow{#1}} 

\newcommand{\comp}{\circ} 
\newcommand{\rest}[1]{|_{#1}} 
\newcommand{\so}{\mathcal{O}} 
\newcommand{\ring}[1]{\mathbb{#1}}
\newcommand{\K}{\Bbbk} 
\newcommand{\NN}{\ring{N}}
\newcommand{\ZZ}{\ring{Z}}
\newcommand{\FF}{\mathbb{F}}
\newcommand{\PP}{\mathbb{P}}
\newcommand{\cat}[1]{{\mathscr{#1}}} 
\newcommand{\scat}[1]{{\mathbf{#1}}} 
\newcommand{\opp}{^{\circ}} 
\newcommand{\End}{{\rm End}} 
\newcommand{\cone}[1]{\mathsf{C}(#1)} 
\newcommand{\farg}{-} 
\newcommand{\fMod}[1]{#1\text{-}\mathrm{mod}} 
\newcommand{\proj}[1]{#1\text{-}\mathrm{proj}} 
\newcommand{\gen}[1]{\langle#1\rangle} 
\newcommand{\sh}[2][1]{#2[#1]} 
\newcommand{\tri}[6]{#1\mor{#2}#3\mor{#4}#5\mor{#6}\sh{#1}}
\newcommand{\Pretr}{\mathrm{Pretr}} 
\newcommand{\C}[1]{\mathrm{C}_{\mathrm{dg}}(#1)} 

\newcommand{\ca}{\cat A}
\newcommand{\cb}{\cat B}

\newcommand{\ce}{\cat E}

\newcommand{\cs}{\cat S}
\newcommand{\ct}{\cat T}
\newcommand{\D}{\scat D} 
\newcommand{\Da}{\D^?}

\newcommand{\Db}{\D^b}
\newcommand{\Dbz}{\D_0^b}
\newcommand{\Kb}{\scat K^b}
\newcommand{\Hqe}{\scat{Hqe}}
\newcommand{\Hqep}{\Hqe^{\mathrm{p}}}
\newcommand{\He}{\scat{He}} 
\newcommand{\Ht}{\scat{Ht}} 
 
\newcommand{\Mor}{\scat{Mor}}

\newcommand{\E}{E}
\newcommand{\I}{I}
\newcommand{\J}{J}

\newcommand{\lotimes}{\mathbin{\mathop{\otimes}\limits^{\mathbb{L}}}} 
\newcommand{\R}{\mathbb{R}} 
\newcommand{\fun}[1]{\mathsf{#1}} 
\newcommand{\fE}{\fun{E}}
\newcommand{\fF}{\fun{F}}
\newcommand{\fG}{\fun{G}}

\newcommand{\fR}{\fun{R}}
\newcommand{\fS}{\fun{S}}

\newcommand{\dgCat}{\mathbf{dgCat}} 
\newcommand{\stable}[1]{\underline{#1}}

\date{\today}

\begin{document}

\title{New examples of non-Fourier-Mukai exact functors via non-isomorphic octahedra}

\author[A.~Canonaco]{Alberto Canonaco}
\address{A.C.: Dipartimento di Matematica ``F.\ Casorati''\\
Universit{\`a} degli Studi di Pavia\\
Via Ferrata 5\\
27100 Pavia\\
Italy}
\email{alberto.canonaco@unipv.it}
    
\author[M.~Ornaghi]{Mattia Ornaghi}
\address{M.O.: Dipartimento di Matematica ``F.~Enriques''\\Universit\`a degli Studi di Milano\\Via Saldini 50\\ 20133 Milano\\ Italy}
\email{mattia12.ornaghi@gmail.com}
\urladdr{\url{https://sites.google.com/view/mattiaornaghi}}

\thanks{A.~C.~is members of GNSAGA (INdAM) and was partially supported by the research project PRIN 2022 ``Moduli spaces and special varieties''.
M.~O.~is supported by the ERC Advanced Grant 101095900-TriCatApp.}



\begin{abstract}
We study a triangulated category $\cs$ that admits a full and strong exceptional sequence of three objects with one-dimensional Hom spaces. We show that the isomorphism classes of exact functors from $\cs$ to another triangulated category $\ct$ are in bijection with the isomorphism classes of octahedra in $\ct$ satisfying a natural condition. As an application, we construct an exact functor from $\cs$ to $\Db(\fMod{\K[x]})$ that does not admit a dg lift. This provides an explicit example of a non-Fourier-Mukai exact functor between $\Db(\PP^2)$ and $\Db(\PP^1)$.
\end{abstract}

\maketitle

\section*{Introduction}\label{intro}

Let $X_1$ and $X_2$ be two smooth projective schemes over a field $\K$. An exact ($\K$-linear) functor between the bounded derived categories of coherent sheaves $\fF\colon\Db(X_1)\to\Db(X_2)$ is of {\em Fourier-Mukai} type (or simply {\em Fourier-Mukai}) if there exists $E\in\Db(X_1\times X_2)$ such that $\fF\iso\Phi_E$, where
\[
\Phi_E(\farg):=\R(p_2)_{\ast}\big( E \lotimes p_1^{\ast}(\farg) \big)\colon \Db(X_1)\to \Db(X_2),
\]
with $p_i:X_1\times X_2\to X_i$ the natural projection, for $i=1,2$. The object $E$ is called a {\em kernel} of $\fF$. Such functors were introduced by Shigeru Mukai in the 1980s to study the equivalence between the derived categories of an abelian variety and of its dual. A decade later Orlov proved a celebrated result, which goes under the name \emph{Orlov's representability theorem}: if $\fF$ is fully faithful, then it is Fourier-Mukai with a unique kernel, up to isomorphism (see \cite[2.2 Theorem]{Or2}). To be precise, the original statement also required the existence of an adjoint to $\fF$, but this hypothesis turned out to be redundant. Indeed, it is an easy consequence of \cite[Theorem 1.1]{BV} that every exact functor $\Db(X_1)\to\Db(X_2)$ has adjoints.

Motivated by Orlov's theorem and by the fact that all meaningful geometric functors $\Db(X_1)\to\Db(X_2)$ are Fourier-Mukai, Bondal, Larsen and Lunts conjectured that every exact functor $\Db(X_1)\to\Db(X_2)$ is Fourier-Mukai (see \cite[Conjecture 7.2]{BLL}). It took several years to disprove the conjecture: the first example of a non-Fourier-Mukai exact functor was found in \cite[Theorem 1.4]{RVN} by Rizzardo, Van den Bergh and Neeman. Their construction relies on formidable technical machinery, and even the definition of the functor itself is difficult to access for non-experts (see \cite[6.2]{CS} for a concise overview). 
Subsequent work by the same authors and their collaborators has led to generalizations and partial semplifications of this example (see \cite[Theorem 1.3]{RRV}, \cite{Ku1}). Nevertheless, all known examples of this type share common features: they rely on a refined obstruction theory involving highly non-trivial $A_{\infty}$ methods and require $\dim(X_1),\dim (X_2)\ge3$. Indeed, the techniques developed in \cite{RVN} simply do not work in dimension $\le2$. The same seems to be true for the only other known examples of non-Fourier-Mukai exact functors, namely those obtained by Vologodsky in \cite{V}. Although these examples are much easier to describe, what makes them less interesting is the fact that they require $\K=\FF_p$ (where $p$ is a prime number).

It is therefore natural to ask whether there exist `easy' examples of non-Fourier-Mukai exact functors over an arbitrary field, possibly with $\dim(X_1),\dim(X_2)\le2$. In particular, the cases $X_1=\PP^1$ or $X_2=\PP^1$ are intriguing. Indeed, the derived category $\Db(\PP^1)$ admits a very explicit description, and when $X_1=\PP^1$ or $X_2=\PP^1$ every Fourier-Mukai functor has a kernel that is unique up to isomorphism, see \cite[Proposition 2.4]{CS1}. The main result of this paper is the following.

\begin{thmInt}\label{thmgeo}
There exist non-Fourier-Mukai exact functors $\Db(\PP^2)\to\Db(\PP^1)$ and $\Db(\PP^1)\to\Db(\PP^2)$.
\end{thmInt}

This result holds over any field and is an easy geometric application of a purely algebraic theorem. Before stating it, we recall that an exact functor $\Db(X_1)\to\Db(X_2)$ is Fourier-Mukai if and only if it has a (dg) lift. It is also important to note that the category $\Db(X_i)$ (for $i=1,2$) admits a strongly unique enhancement. As a consequence, the existence of a lift for a given functor is independent of the choice of enhancements of $\Db(X_i)$ (see \autoref{dglifts} for more details). We now state our main algebraic result. Throughout, we denote by $\fMod A$ the category of finitely generated left modules over a $\K$-algebra $A$.

\begin{thmInt}\label{thmalg}
Let $\cs$ be a ($\K$-linear) triangulated category with a full and strong exceptional sequence $(A,B,C)$ such that $\cs(A,B)\iso\cs(A,C)\iso\cs(B,C)\iso\K$ and the composition map $\cs(A,B)\otimes\cs(B,C)\to\cs(A,C)$
is nonzero. Then there exists an exact functor $\cs\to\Db(\fMod{\K[x]})$ which does not admit a lift. Moreover, the essential image of the functor is contained in the full triangulated subcategory $\Dbz(\fMod{\K[x]})$ of $\Db(\fMod{\K[x]})$ consisting of complexes whose cohomology is annihilated by some power of $x$.
\end{thmInt}

We point out that \autoref{thmalg} provides an example of a non liftable functor between the bounded derived categories of two hereditary algebras. Indeed, one has $\cs\iso\Db(\fMod A)$ where $A$ is a (left) hereditary algebra (see \autoref{KS}). Note that such triangulated categories have strongly unique enhancements; see \autoref{hered}).

We now briefly describe the proof of \autoref{thmalg}. The triangulated category $\cs$ has a particularly simple structure, which allows a complete characterization of its distinguished triangles (see \autoref{triangles}). More precisely, we denote by $f$ and $g$ the generators of the one-dimensional $\K$-modules $\cs(A,B)$ and $\cs(B,C)$, respectively. Then every distinguished triangle in $\cs$ arises from the four distinguished triangles appearing in an octahedron extending $f$ and $g$, together with the two additional (a priori not necessarily distinguished) triangles associated with this octahedron (see \cite[Remarque 1.1.13]{BBD}). In general, we call an octahedron \emph{good} if these two extra triangles are distinguished. It follows that isomorphism classes of exact functors from $\cs$ to a triangulated category $\ct$ are in bijection with isomorphism classes of good octahedra in $\ct$; see \autoref{cor:bij}. This characterization makes it easy to deduce the existence of an exact functor $\cs\to\ct$ without a lift, whenever $\ct$ admits a strongly unique enhancement and contains two non isomorphic good octahedra extending the same pair of composable morphisms (see \autoref{nolift}). To the best of our knowledge, the only example in the literature of two non-isomorphic (good) octahedra extending the same pair of composable morphisms is due to K\"unzer, see \cite[\S 2]{Ku}. In his example, the triangulated category is the stable Frobenius category of $\fMod{\ZZ/(p^6)}$, which is not linear over a field. A $\K$-linear variant can be obtained by replacing the ring $\ZZ/(p^6)$ with $\K[x]/(x^6)$. However, it is unclear whether this triangulated category has a (strongly) unique enhancement. On the other hand, we show that K\"unzer's construction can be adapted to produce a similar example in $\Dbz(\fMod{\K[x]})$ (see \autoref{prop:oct}).

\subsection*{Conventions and content of the paper}

Throughout the paper $\K$ denotes a field (although many general results hold when $\K$ is an arbitrary commutative ring). All preadditive categories and all additive functors are assumed to be $\K$-linear. 

In \autoref{dglifts} we recall basic notions concerning dg categories, (dg) enhancements of triangulated categories and (dg) lifts of exact functors, together with the relationship between liftable and Fourier-Mukai functors in the geometric setting. All of this material is standard, except for some notational conventions. In \autoref{sec:goodoct} we introduce the notion of good octahedron and we provide an explicit example of two non-isomorphic good octahedra extending the same pair of composable morphisms; see \autoref{prop:oct}. A complete proof of this proposition, closely following K\"unzer's original argument, is given in \autoref{noniso}. In \autoref{sec:Scat}, after presenting some general properties of triangulated categories admitting a full and strong exceptional sequence, we treat the case of the category $\cs$ in \autoref{thmalg}, for which we state \autoref{triangles}. Although this result is possibly known to experts, we give a proof in \autoref{excseq}, as we were unable to find a suitable reference in the literature. Finally, in \autoref{sec:thmalg} and \autoref{sec:thmgeo} we prove \autoref{thmalg} and \autoref{thmgeo}.

\section{Dg lifts and dg enhancements}\label{dglifts}
We assume the reader is familiar with the language of dg categories. A complete introduction can be found in \cite{K2}, see also \cite{Ge}. 

We denote by $\dgCat$ the category of (small, $\K$-linear) dg categories with dg functors as morphisms, and by $\Hqe$ its localization with respect to quasi-equivalences. Note that $\dgCat$ and $\Hqe$ have the same objects. Regarding morphisms, we recall that $\dgCat$ has a model structure whose weak equivalences are quasi-equivalences and such that every object is fibrant (see \cite{Ta}). It follows that every morphism $\ca\to\cb$ in $\Hqe$ can be represented by a diagram $\ca\xleftarrow\fS\ca'\mor\fF\cb$ in $\dgCat$, where $\fS$ is any quasi-equivalence with $\ca'$ cofibrant. 

Given a dg category $\ca$, an object $A\in\ca$ and a closed degree zero morphism $f$ in $\ca$, we can always define the $n$-shift of $A$ (where $n\in\ZZ$) and the mapping cone of $f$ (see \cite[Definition 2.3.1 and 2.3.2]{Ge}). 
We say that $\ca$ is {\em strongly pretriangulated} if it is closed under shifts and cones. 
Given a dg category $\ca$ we can always ``embed'' $\ca$ in a strongly pretriangulated dg category.
Namely, we denote by $\Pretr(\ca)$ the {\em pretriangulated envelope} of $\ca$. This is the smallest strongly pretriangulated full dg subcategory of $\C{\ca}$, containing the image of $\ca$ under the Yoneda embedding, which is closed under dg isomorphisms.
Here $\C{\ca}$ denotes the dg category of dg functors from the opposite dg category $\ca\opp$ to the dg category of complexes of $\K$-modules $\C{\K}$. We say that a dg category $\ca$ is {\em pretriangulated} if the Yoneda embedding $\ca\to \Pretr(\ca)$ is a quasi-equivalence.  Denoting by $\Hqep$ the full subcategory of $\Hqe$ whose objects are the pretriangulated dg categories, the pretriangulated envelope gives rise to a functor $\Pretr\colon\Hqe\to \Hqep$, which is left adjoint to the inclusion functor $\Hqep\to\Hqe$ (see \cite[\S 4.5]{K2} or \cite[Proposition 3.5.3]{Ge}).

If $\ca$ is a pretriangulated dg category then its homotopy category $H^0(\ca)$ has a natural structure of triangulated ($\K$-linear) category. We will denote by $H^0_t(\ca)$ the category $H^0(\ca)$ with this triangulated structure. As a matter of notation, $\He$ denotes the category of (small, $\K$-linear) categories whose morphisms are the isomorphism classes of functors.
On the other hand, $\Ht$ denotes the category of (small, $\K$-linear) triangulated categories whose morphisms are the isomorphism classes of exact functors. 

\begin{remark}
While a pretriangulated dg category is just a dg category satisfying some properties, a triangulated category is given by a category with additional structure (shift and distinguished triangles). 
An exact functor between triangulated categories must preserve the triangulated structure. In particular, it is given by a pair $(\fF,\alpha)$, where $\fF$ is an ordinary functor and $\alpha\colon\fF\comp\sh{} \to\sh{}\comp\fF$ is a natural isomorphism. With a standard abuse of notation, we will denote such an exact functor simply by $\fF$. However, it is important to recall that an (iso)morphism of exact functors $\gamma\colon(\fF,\alpha)\to(\fF',\alpha')$ is given by an ordinary (iso)morphism of functors $\gamma\colon\fF\to\fF'$ such that $\alpha'_X\comp\gamma_{\sh X}=\sh{\gamma_X}\comp\alpha_X$, for every object $X$ in the source category.
\end{remark}

Clearly the map sending a dg category (respectively a pretriangulated dg category) $\ca$ to $H^0(\ca)$ (respectively $H^0_t(\ca)$) extends to a functor $H^0\colon\Hqe\to\He$ (respectively $H^0_t\colon\Hqep\to\Ht$). Moreover, there is a commutative diagram 
\begin{equation}\label{eq:comm}
\xymatrix{
\Hqep\ar[d]_{H^0_t} \ar[r] & \Hqe\ar[d]^{H^0} \\
\Ht \ar[r] & \He,
}
\end{equation}
where the top (respectively bottom) horizontal map is the natural inclusion (respectively forgetful) functor.

\begin{definition}
A {\em (dg) enhancement} of a triangulated category $\ct$ is a pair $(\ca,\fE)$, where $\ca$ is a pretriangulated dg category and $\fE\colon H^0_t(\ca)\to \ct$
is an isomorphism in $\Ht$. A triangulated category having an enhancement is called {\em algebraic}.

Let $\ct$ be an algebraic triangulated category. We say that $\ct$ has a {\em (strongly) unique enhancement} if, given two enhancements $(\ca,\fE)$ and $(\ca',\fE')$ of $\ct$, there exists an isomorphism $f\colon \ca\to \ca'$ in $\Hqe$ (such that $\fE=\fE'\comp H^0_t(f)$ in $\Ht$).
\end{definition}

\begin{ex}\label{hered}
If $\ca$ is a hereditary abelian category and $?\in\{-,+,b,\emptyset \}$ then $\Da(\ca)$ has a strongly unique enhancement, see \cite[Theorem A]{CNS2}.
\end{ex}

\begin{definition}\label{def:lift}
Let $(\ca,\fE)$ and $(\ca',\fE')$ be enhancements of two triangulated categories $\ct$ and $\ct'$. We say that 
$f\in\Hqe(\ca,\ca')$ is a {\em (dg) lift} of $\fF\in\Ht(\ct,\ct')$, with respect to the enhancements $(\ca,\fE)$ and $(\ca',\fE')$, if the following diagram
\[
\xymatrix{
H^0_t(\ca) \ar[d]_{\fE} \ar[r]^-{H^0_t(f)} & H^0(\ca') \ar[d]^{\fE'} \\
\ct \ar[r]^-{\fF} & \ct'
}
\]
is commutative in $\Ht$.
\end{definition}

\begin{remark}\label{stronglift}
In general, whether $\fF$ admits a lift depends on the choice of the enhancements $(\ca,\fE)$ and $(\ca',\fE')$.
However, it is easy to see that this dependence disappears when $\ct$ and $\ct'$ have strongly unique enhancements; in this case we can simply say that 
$\fF$ has a lift.
\end{remark}

\begin{prop}\label{FMlift}
Let $X_1$ and $X_2$ be two smooth projective schemes over $\K$. 
\begin{enumerate}
\item $\Db(X_i)$ has a strongly unique enhancement, for $i=1,2$.
\item An exact functor $\Db(X_1)\to\Db(X_2)$ is Fourier-Mukai if and only if it has a lift.
\end{enumerate}
\end{prop}

\begin{proof}
One can choose suitable enhancements of $\Db(X_1)$ and $\Db(X_2)$ in such a way that an exact functor $\Db(X_1)\to\Db(X_2)$ is Fourier–Mukai if and only if it admits a lift. This fact follows from a more general statement claimed by To\"en (see \cite[\S 8.3]{To}).
A complete proof of To\"en's statement in full generality will be available in \cite{COS}. However, note that the claim was already proved under stronger hypotheses by Lunts and Schn\"urer in \cite[Theorem 1.1]{LS}, and this is enough in our setting. Then one can conclude easily using Orlov's representability theorem (see also \cite[\S 6.3]{CS}).
\end{proof}

\section{non-isomorphic good octahedra}\label{sec:goodoct}

Let $\ct$ be a triangulated category.
Let $A\mor f B\mor g C$ be two composable morphisms of $\ct$ and $h:=g\comp f$. We can complete $f$, $g$ and $h$ to three distinguished triangles
\begin{gather}
\label{eq:tri1}
\tri A f B{f'}D{f''}, \\
\label{eq:tri2}
\tri A hC{h'}E{h''}, \\
\label{eq:tri3}
\tri B g C{g'}F{g''},
\end{gather}
in $\ct$.
By (TR4) there exists a commutative diagram
\begin{equation}\label{eq:oct}
\xymatrix{
A \ar[r]^-f \ar@{=}[d] & B \ar[r]^-{f'} \ar[dr]^-l \ar[d]^g & D \ar[r]^-{f''} \ar[d]^k & \sh A \ar@{=}[d] \\
A \ar[d]^f \ar[r]^-h & C \ar[r]^-{h'} \ar@{=}[d] & E \ar[r]^-{h''} \ar[d]^{k'} \ar[dr]^-{l'} & \sh A \ar[d]^{\sh f} \\
B \ar[r]^-g & C \ar[r]^-{g'} & F \ar[r]^-{g''} & \sh B
}
\end{equation}
such that 
\begin{gather}
\label{eq:tri4}
\tri D k E{k'}F{k''}
\end{gather}
is a distinguished triangle in $\ct$, where $k'':=\sh{f'}\comp g''$.
An \emph{octahedron} in $\ct$ is a commutative diagram of the form \eqref{eq:oct} such that \eqref{eq:tri1}, \eqref{eq:tri2}, \eqref{eq:tri3} and \eqref{eq:tri4} are distinguished triangles. To shorten notation, we set
\begin{equation}\label{eq:sets}
\I:=\{A,B,C,D,E,F\}, \qquad
\J:=\{f,f',f'',g,g',g'',h,h',h'',k,k',k'',l,l'\}.
\end{equation}
There is an obvious notion of isomorphism for octahedra. Namely, assume that we have two octahedra, the first given by \eqref{eq:oct}, and the second by a similar diagram, where every object $X\in\I$ is replaced by $\tilde X$ and every morphism $x\in\J$ is replaced by $\tilde x$. Then we say that the two octahedra are isomorphic if there exist six isomorphisms $\phi_X\colon X\to\tilde X$, for $X\in\I$, such that $\tilde x\comp\phi_Y=\phi_Z\comp x$ for every morphism $x\colon Y\to Z$ in $J$, where, if $Z=\sh X$ for some $X\in\I$, we set $\phi_Z:=\sh{\phi_X}$.

We will say that the octahedron \eqref{eq:oct} is \emph{good} if the following triangles
\begin{gather}
\label{eq:tri5}
\tri B l E{
\bigl(\begin{smallmatrix}
k' \\
-h''
\end{smallmatrix}\bigl)
}{F\oplus\sh A}{
(
\begin{smallmatrix}
g'' & \sh f    
\end{smallmatrix})
}, \\
\label{eq:tri6}
\tri B{
\bigl(
\begin{smallmatrix}
g \\
f'
\end{smallmatrix}
\bigr)
}{C\oplus D}{
(\begin{smallmatrix}
h' & -k    
\end{smallmatrix})
}E{l'}.
\end{gather}
are distinguished.

\begin{remark}\label{goodcrit}
It seems to be an open problem if in an arbitrary triangulated category every pair of composable morphisms can be completed to a good octahedron (see \cite[Remarque 1.1.13]{BBD}). For later use, we point out a (rather restrictive) sufficient condition for an octahedron to be good. Note that in any case, by (TR4) applied to $l=h'\comp g$, there exists an octahedron
\[
\xymatrix{
B \ar[r]^-g \ar@{=}[d] & C \ar[r]^-{g'} \ar[d]^{h'} & F \ar[r]^-{g''} \ar[d]^c & \sh B \ar@{=}[d] \\
B \ar[d]^g \ar[r]^-l & E \ar[r]^-a \ar@{=}[d] & G \ar[r]^-b \ar[d]^d & \sh B \ar[d]^{\sh g} \\
C \ar[r]^-{h'} & E \ar[r]^-{-h''} & \sh A \ar[r]^-{\sh h} & \sh C.
}
\]
In particular the triangle $\tri A{g'\comp h}F c G d$ is distinguished and, since $g'\comp h=g'\comp g\comp f=0$, we can assume (up to replacing $a$, $b$, $c$ and $d$ with $s\comp a$, $b\comp s^{-1}$, $s\comp c$ and $d\comp s^{-1}$ for some isomorphism $s$) $G=F\oplus\sh A$, with $c$ the natural inclusion and $d$ the natural projection. Writing 
\[
a=\bigl(\begin{smallmatrix}
a' \\
a''
\end{smallmatrix}\bigr)
\colon E\to F\oplus\sh A, \qquad
b=(\begin{matrix}
b' & b''
\end{matrix})
\colon F\oplus\sh A\to\sh B,
\]
the commutativity of the diagram implies
\[
b'=g'', \qquad a''=-h'', \qquad a'\comp h'=g', \qquad \sh g\comp b''=\sh h.
\]
Therefore it is certainly true that \eqref{eq:tri5} is distinguished if $k'\colon E\to F$ is the only morphism such that $k'\comp h'=g'$ and $\sh f\colon\sh A\to\sh B$ is the only morphism such that $\sh g\comp\sh f=\sh h$. Similarly one can see that \eqref{eq:tri6} is distinguished if $g\colon B\to C$ is the only morphism such that $g\comp f=h$ and $k\colon D\to E$ is the only morphism such that $h''\comp k=f''$.
\end{remark}

We conclude this section by giving an example of two non-isomorphic good octahedra in $\Dbz(\fMod R)$, where $R:=\K[x]$.
Setting $R_n:=R/(x^n)$ (for $n\in\NN$), we denote by $x^i\colon R_m\to R_n$, for $i\ge0,n-m$, the morphism of $\fMod R$ induced by multiplication by $x^i$. Moreover, $y\colon R_m\to\sh{R_n}$ will denote the morphism of $\Db(\fMod R)$ corresponding to the isomorphism class of the short exact sequence 
\[
0\to R_n\mor{x^m}R_{n+m}\mor{x^0} R_m\to0
\]
in $\fMod R$.

\begin{prop}\label{prop:oct}
The two composable morphisms
\begin{equation}\label{eq:setup1}
A:=R_3\mor{f:=x^1}B:=R_3\mor{g:=x^1}C:=R_3
\end{equation}
can be extended to two non-isomorphic good octahedra in $\Dbz(\fMod R)$.
Using the notation of \eqref{eq:oct}, the first octahedron is given by:
\begin{equation}\label{eq:oct3}
\begin{gathered}
F=D=\sh{R_1}\oplus R_1, \qquad E=\sh{R_2}\oplus R_2,\\
f'=\Bigl(\begin{smallmatrix}
y \\
x^0
\end{smallmatrix}\Bigr),
\qquad h'=\Bigl(\begin{smallmatrix}
y \\
x^0
\end{smallmatrix}\Bigr),
\qquad g''=(\begin{smallmatrix}
-\sh{x^2} &\quad y
\end{smallmatrix}), \qquad h''=(\begin{smallmatrix}
-\sh{x^1} &\quad y
\end{smallmatrix}), 
\end{gathered}
\end{equation}
and 
\begin{gather}\label{eq:oct3b}
k=\Bigl(\begin{smallmatrix}
\sh{x^1} & 0\\
0 & x^1
\end{smallmatrix}\Bigr), \qquad k'=\Bigl(\begin{smallmatrix}
\sh{x^0} & 0\\
0 & x^0
\end{smallmatrix}\Bigr)
\end{gather}
(the other morphisms are determined by composition). The second octahedron is given by \eqref{eq:oct3} and
\begin{equation}\label{eq:octdata2}
\tilde{k}=\Big(\begin{smallmatrix}
\sh{x^1} & y \\
0 & x^1
\end{smallmatrix}\Bigr), \qquad \tilde{k}'=\Bigl(\begin{smallmatrix}
\sh{x^0} & -y \\
0 & x^0
\end{smallmatrix}\Bigr).
\end{equation}
\end{prop}

\begin{remark}\label{Kunzer}
The above example is closely related to the one by K\"unzer in \cite{Ku}. More precisely, all objects and morphisms in \autoref{prop:oct} can be interpreted also in the stable Frobenius category $\stable{\fMod{R_6}}$ of $\fMod{R_6}$ (to see this, one has to use the fact that $\sh{R_i}\iso R_{6-i}$ in $\stable{\fMod{R_6}}$, for $0\le i\le6$). If one reformulates K\"unzer's example in an obvious way, replacing $\stable{\fMod{\ZZ/(p^6)}}$ with $\stable{\fMod{R_6}}$, it turns out that the first octahedron in \autoref{prop:oct} corresponds exactly to the first octahedron in \cite[\S 2]{Ku}. Since the same is not true for the second octahedron, we cannot directly deduce \autoref{prop:oct} from K\"unzer's result. However, the proof of \autoref{prop:oct}, which will be given in \autoref{noniso}, is much inspired by \cite{Ku}.
\end{remark}

\section{Full and strong exceptional sequences}\label{sec:Scat}
Let $\cs$ be a triangulated category with a full and strong exceptional sequence $(\E_1,\dots,\E_n)$. Recall that this means the following, for every $i,j\in\{1,\dots,n\}$ and every $n\in\ZZ$:
\begin{itemize}
    \item $\cs(\E_i,\E_i)\iso\K$;
    \item $\cs(\E_i,\sh[n]{\E_j})\ne0\implies n=0\text{ and }i\le j$;
    \item $\cs=\gen{\E_1,\dots,\E_n}$ is the triangulated envelope of $\{\E_1,\dots,\E_n\}$.
\end{itemize}
We will denote by $\ce$ the full subcategory of $\cs$ with objects $\{\E_1,\dots,\E_n\}$, and we will also assume that $\ce$ has finite-dimensional (over $\K$) Homs. Note that this implies that $\oplus_{n\in\ZZ}\cs(X,\sh[n]Y)$ is finite-dimensional for every $X,Y\in\cs$.

Before stating the next result, we recall that, given a quiver $Q$, the category $\K[Q]$ is defined as follows:
\begin{itemize}
\item $\K[Q]$ has the same objects of $Q$.
\item Given $X,Y\in \K[Q]$, the Hom space $\K[Q](X,Y)$ is the free $\K$-module generated by all finite strings of composable morphisms (of $Q$) with (initial) source $X$ and (final) target $Y$.
\item The composition is given by string concatenation, extended by bilinearity. The identities are the empty strings.
\end{itemize}

\begin{lem}\label{KS}
$\cs$ is a Krull-Schmidt category.\footnote{An additive category is Krull-Schmidt if every object is a finite direct sum of objects with local endomorphism rings. In a Krull-Schmidt category an object is indecomposable if and only if its endomorphism ring is local, and every object is in an essentially unique way a finite direct sum of indecomposable objects.} Moreover, if $\cs$ is algebraic, then the following is true.
\begin{enumerate}
    \item $\cs\iso\Db(\fMod A)$ in $\Ht$, where $A:=\End_\cs(\oplus_{i=1}^n\E_i)$.
    \item $\cs$ has a strongly unique enhancement and an enhancement of $\cs$ is given by $\Pretr(\ce)$ (where $\ce$ is regarded as a dg category concentrated in degree $0$).
    \item If $\ce$ is of the form $\K[Q]$ for some quiver $Q$, then $A$ is left hereditary.
\end{enumerate}
\end{lem}

\begin{proof}
It follows immediately from \cite[Lemma 4.8]{BDFIK} that $\cs$ is idempotent complete. Moreover, $\End_\cs(X)$ is a semi-perfect ring (being finite-dimensional over $\K$) for every $X\in\cs$. Hence $\cs$ is a Krull-Schmidt category by \cite[Corollary 4.4]{Kr}.

Assume now that $\cs$ is algebraic.
\begin{enumerate}
    \item See, for instance, \cite[Corollary 1.9]{Or}.
    \item It is a standard fact (a proof can be found in \cite[Lemma 4.3.1 and Remark 4.3.2]{Ge}) that $\Pretr(\ce)$ is an enhancement of $\cs$. Finally, $\cs$ has a strongly unique enhancement by \cite[Theorem 5.61]{Lo1}.\footnote{When $A$ is left (semi)hereditary (so that $\fMod A$ is a hereditary category), one can alternatively invoke \autoref{hered}.}
    \item It is easy to see that $A$ is isomorphic (as a $\K$-algebra) to the path $\K$-algebra of $Q$. Then $A$ is left hereditary by \cite[VII, Theorem 1.7]{ASS}.
\end{enumerate}
\end{proof}

We now assume that $\ce$ is of the form $\K[Q]$, where $Q$ is the quiver $A\mor f B\mor g C$. More explicitly, $n=3$, we write $A$, $B$, $C$ instead of $\E_1$, $\E_2$ and $\E_3$, and
\begin{equation}\label{eq:S}  
\cs(A,B)=\K f, \qquad \cs(B,C)=\K g, \qquad \cs(A,C)=\K h,
\end{equation}
where $h:=g\comp f\ne0$. Note that $\cs$ is as in the statement of \autoref{thmalg}. 

\begin{prop}\label{triangles}
Let $\cs$ be a triangulated category with a full and strong exceptional sequence as in \eqref{eq:S}. Then $\cs$ is algebraic and is uniquely determined, up to isomorphism in $\Ht$. Moreover, extending $A\mor f B\mor g C$ to an octahedron, for which we use the same notation of \eqref{eq:oct} and \eqref{eq:sets}, the following holds.
\begin{enumerate}
\item\label{triangles1} Up to isomorphism, every indecomposable object of $\cs$ is of the form $\sh[n]X$, for a unique $X\in\I$ and a unique $n\in\ZZ$. 
\item\label{triangles2} The indecomposable objects of $\cs$ are exceptional and, up to shift and isomorphism, the only nonzero morphisms between non-isomorphic indecomposable objects are those in $\J$. More precisely, given distinct $X,Y\in\I$ and $n\in\ZZ$, we have $\cs(X,\sh[n]Y)=0$ except in the following cases:
\begin{gather*}
\cs(A,B)=\K f, \qquad \cs(A,C)=\K h, \qquad \cs(B,C)=\K g, \qquad \cs(B,D)=\K f', \\
\cs(B,E)=\K l, \qquad \cs(C,E)=\K h', \qquad \cs(C,F)=\K g', \qquad \cs(D,E)=\K k, \\
\cs(D,\sh A)=\K f'', \qquad \cs(E,F)=\K k', \qquad \cs(E,\sh A)=\K h'', \\
\cs(E,\sh B)=\K l', \qquad \cs(F,\sh B)=\K g'', \qquad \cs(F,\sh D)=\K k''.
\end{gather*}
The composition of morphisms is completely determined by the following relations:
\begin{equation}\label{eq:comp}
\begin{split}
f''=h''\comp k, \qquad g'=k'\comp h', \qquad h=g\comp f, \qquad k''=\sh{f'}\comp g'', \qquad l=h'\comp g=k\comp f', \\
l'=\sh f\comp h''=g''\comp k', \qquad f'\comp f=0, \qquad h''\comp h'=0, \qquad \sh g\comp g''=0, \qquad k'\comp k=0.
\end{split}
\end{equation}
\item\label{triangles3} A triangle in $\cs$ is distinguished if and only if it is (isomorphic to) a finite direct sum of triangles, each of which, up to rotation, is either one among \eqref{eq:tri1}, \eqref{eq:tri2}, \eqref{eq:tri3}, \eqref{eq:tri4}, \eqref{eq:tri5}, \eqref{eq:tri6}(in particular, the octahedron \eqref{eq:oct} is good), or  $\tri X{\id}X{}0{}$ for some $X\in\I$ .
\end{enumerate}
\end{prop}

As a direct consequence of \autoref{triangles} (whose proof will be given in \autoref{excseq}) we obtain the following result.

\begin{cor}\label{cor:bij}
Let $\ct$ be a (small) triangulated category. There is a natural bijection between $\Ht(\cs,\ct)$ and the set of isomorphism classes of good octahedra in $\ct$.
\end{cor}

\section{Proof of \autoref{thmalg}}\label{sec:thmalg}

We start with a result providing a sufficient condition for the existence of an exact functor with no lift.

\begin{lem}\label{noliftcrit}
Let $\ca$ be a dg category and let $\ce:=\K[Q]$, where $Q$ is a quiver. 
\begin{enumerate}
\item The map $H^0\colon\Hqe(\ce,\ca)\to\He(\ce,H^0(\ca))$ is injective.
\end{enumerate}
Suppose that $\ca$ is pretriangulated and that $H^0(\Pretr(\ce))$ and $H^0(\ca)$ have strongly unique enhancements.
\begin{enumerate}
\setcounter{enumi}{1}
\item If $\fF,\fF'\in\Ht(H^0(\Pretr(\ce)),H^0(\ca))$ are such that $\fF\ne\fF'$ and $\fF|_{\ce}=\fF'|_{\ce}\in\He(\ce,H^0(\ca))$, then (at least) one among $\fF$ and $\fF'$ does not have a lift.
\end{enumerate}
\end{lem}

\begin{proof}
\begin{enumerate}
\item One can use the same argument as in the proof of \cite[Theorem 4.3.8]{Ge}. More precisely, using the fact that $\K[Q]$ is cofibrant in $\dgCat$ (see \cite[Proposition 4.3.5]{Ge}), every element of $\Hqe(\ce,\ca)$ can be represented by a dg functor $\ce\to\ca$. Then one can apply \cite[Lemma 4.3.7]{Ge}.
\item Suppose that $\fF$ has a lift. In light of \autoref{stronglift}, this amounts to saying that $\fF=H^0_t(f)$ for some $f\in\Hqep(\Pretr(\ce),\ca)$. Recalling \eqref{eq:comm}, there is a commutative diagram
\[
\xymatrix{
\Hqep(\Pretr(\ce),\ca) \ar[d]^{H^0_t} \ar[r]_{\iso} & \Hqe(\Pretr(\ce),\ca) \ar[r]^-{(\farg)\rest{\ce}}_-{\iso} & \Hqe(\ce,\ca) \ar@{^(->}[d]^{H^0}\\
\Ht(H^0(\Pretr(\ce)),H^0(\ca)) \ar[r] \ar[r] & \He(H^0(\Pretr(\ce)),H^0(\ca)) \ar[r]^-{(\farg)\rest{\ce}} & \He(\ce,H^0(\ca)),
}
\]
where the upper row is a bijection due to the adjunction between $\Pretr$ and the inclusion functor $\Hqep\to\Hqe$, while $H^0$ is injective by the first part. Then $f$ is the unique element of $\Hqep(\Pretr(\ce),\ca)$ which is mapped to $\fF|_{\ce}=\fF'|_{\ce}\in\He(\ce,H^0(\ca))$. Therefore, there is no $f'\in\Hqep(\Pretr(\ce),\ca)$ such that $H^0_t(f')=\fF'$, and this proves that $\fF'$ has no lift.
\end{enumerate}
\end{proof}

\begin{prop}\label{nolift}
Let $\cs$ be as in the statement of \autoref{thmalg} and let $\ct$ be a triangulated category with a strongly unique enhancement. If there exist two non isomorphic good octahedra extending the same pair of composable morphisms in $\ct$, then there exists an exact functor $\cs\to\ct$ without a lift.  
\end{prop}

\begin{proof}
Since $\cs$ is algebraic by \autoref{triangles}, it has a strongly unique enhancement by \autoref{KS}. Recalling also \autoref{stronglift}, we can then assume $\cs=H^0(\Pretr(\ce))$ (where $\ce$ is the full subcategory of $\cs$ with objects $\{A,B,C\}$), and similarly $\ct=H^0(\ca)$, for some pretriangulated dg category $\ca$. By \autoref{cor:bij} the two non isomorphic good octahedra extending the same pair of composable morphisms in $\ct$ correspond to two different elements $\fF,\fF'\in\Ht(\cs,\ct)$ such that $\fF\rest{\ce}=\fF'\rest{\ce}$. Thus, we conclude by \autoref{noliftcrit} that (at least) one among $\fF$ and $\fF'$ does not have a lift.
\end{proof}

Now \autoref{thmalg} follows directly from \autoref{nolift}, taking into account that $\Db(\fMod{\K[x]})$ has a strongly unique enhancement by \autoref{hered} and that there exist two non-isomorphic good octahedra extending the same pair of composable morphisms in $\Dbz(\fMod{\K[x]})$ by \autoref{prop:oct}. 

\begin{remark}
It is not difficult to see that the exact functor corresponding to the octahedron defined by \eqref{eq:oct3} and \eqref{eq:oct3b} has a lift. Hence the octahedron defined by \eqref{eq:oct3} and \eqref{eq:octdata2} does not have a lift.
\end{remark}

\section{Proof of \autoref{thmgeo}}\label{sec:thmgeo}
It is enough to prove that there exists a non-Fourier-Mukai exact functor $\fG\colon\Db(\PP^2)\to\Db(\PP^1)$. Indeed, given such $\fG$, a (left or right) adjoint $\fG'\colon\Db(\PP^1)\to\Db(\PP^2)$ to $\fG$ is also non-Fourier-Mukai (see \cite[Proposition 5.9]{Hu}) and exact (see \cite[1.2 Lemma]{Or}). 

In order to construct $\fG$, we start by considering the functor provided by \autoref{thmalg}. It should be clear that there is an exact equivalence between $\Dbz(\fMod{\K[x]})$ and the full subcategory of $\Db(\PP^1)$ of complexes whose cohomology is supported on a $\K$-rational point. Thus, we obtain an exact functor without a lift $\fF\colon\cs\to\Db(\PP^1)$. Here $\cs$ is as in \autoref{thmalg} and, as before, we let $\ce$ be the full subcategory of $\cs$ with objects $\{A,B,C\}$. We also denote by $\ce'$ the full subcategory of $\Db(\PP^2)$ with objects $\{\so,\so(1),\so(2)\}$. It is well known that $(\so,\so(1),\so(2))$ is a full and strong exceptional sequence in $\Db(\PP^2)$, and it is easy to see that there exist ($\K$-linear) functors $\fS\colon\ce\to\ce'$ and $\fR\colon\ce'\to\ce$ such that $\fR\comp\fS=\id_\ce$. For instance, one can define $\fR$ and $\fS$ as follows. Let $\{a_0,a_1,a_2\}$ and $\{b_0,b_1,b_2\}$ be bases, respectively, of $\Db(\PP^2)(\so,\so(1))$ and $\Db(\PP^2)(\so(1),\so(2))$, where $a_i$ and $b_i$ correspond to multiplication by $x_i$ (regarding $\PP^2$ as $\Proj\K[x_0,x_1,x_2]$). Then we can set
\begin{gather*}
\fS(A):=\so, \qquad \fS(B):=\so(1), \qquad \fS(C):=\so(2), \qquad \fS(f):=a_0, \qquad \fS(g):=b_0, \\
\fR(\so):=A, \qquad \fR(\so(1)):=B, \qquad \fR(\so(2)):=C, \qquad \fR(a_i):=\begin{cases}
f & \text{if $i=0$,} \\
0 & \text{if $i=1,2$,}
\end{cases} \qquad \fR(b_i):=\begin{cases}
g & \text{if $i=0$,} \\
0 & \text{if $i=1,2$.}
\end{cases}
\end{gather*}
By \autoref{KS} we can identify $\cs$ with $H^0(\Pretr(\ce))$ and $\Db(\PP^2)$ with $H^0(\Pretr(\ce'))$. Then we claim that $\fG:=\fF\comp H^0(\Pretr(\fR))$ is an exact functor without a lift, hence non-Fourier-Mukai by \autoref{FMlift}. Indeed, assume on the contrary that $g\in\Hqe(\Pretr(\ce'),\ca)$ (where $\ca$ is an enhancement of $\Db(\PP^1)$) is a lift of $\fG$. From this we obtain the contradiction that $g\comp\Pretr(\fS)\in\Hqe(\Pretr(\ce),\ca)$ is a lift of
\[
\fG\comp H^0(\Pretr(\fS))=\fF\comp H^0(\Pretr(\fR))\comp H^0(\Pretr(\fS))=\fF\comp H^0(\Pretr(\fR)\comp\Pretr(\fS))=\fF,
\]
since obviously $\Pretr(\fR)\comp\Pretr(\fS)=\id_{\Pretr(\ce)}$.

\begin{remark}
It is clear from the proof that \autoref{thmgeo} can be generalized replacing $\PP^1$ with any (smooth projective) curve containing a $\K$-rational point. It is also not difficult to prove that $\PP^2$ can be replaced by any (smooth projective) scheme $X$ such that $\D^b(X)$ contains a (not necessarily full) strong exceptional sequence $(\E'_1,\E'_2,\E'_3)$, provided there exist functors $\fS\colon\ce\to\ce'$ and $\fR\colon\ce'\to\ce$ such that $\fR\comp\fS=\id_\ce$ (where $\ce'$ denotes the full subcategory of $\Db(X)$ with objects $\{\E'_1,\E'_2,\E'_3\}$).
\end{remark}

\appendix \section{Proof of \autoref{triangles}}\label{excseq}

Let $\cs'$ be the full subcategory of $\cs$ whose objects are (finite) direct sums of shifts of objects in $\I$ (eventually we will see that $\cs'=\cs$). 

First we prove that the Hom spaces $\cs(X,\sh[n]Y)$, where $X,Y\in\I$ and $n\in\ZZ$, are as stated in \eqref{triangles2}. By assumption, this is true when $X,Y\in\{A,B,C\}$. Applying the cohomological functor $\cs(A,\farg)$ to the distinguished triangle \eqref{eq:tri1} we obtain an exact sequence
\[
\cs(A,\sh[n]A)\mor{\sh[n]f\comp}\cs(A,\sh[n]B)\to\cs(A,\sh[n]D)\to\cs(A,\sh[n+1]A)\mor{\sh[n+1]f\comp}\cs(A,\sh[n+1]B).
\]
Since $\sh[n]f\comp$ and $\sh[n+1]f\comp$ are isomorphisms, we deduce that $\cs(A,\sh[n]D)=0$. On the other hand, applying $\cs(B,\farg)$ to \eqref{eq:tri1}, we obtain an exact sequence
\[
\cs(B,\sh[n]A)\to\cs(B,\sh[n]B)\mor{\sh[n]{f'}\comp}\cs(B,\sh[n]D)\to\cs(B,\sh[n+1]A).
\]
As $\cs(B,\sh[n]A)=\cs(B,\sh[n+1]A)=0$, we see that $\sh[n]{f'}\comp$ is an isomorphism, and so
\[
\cs(B,\sh[n]D)=\begin{cases}
\K f' & \text{if $n=0$,} \\
0 & \text{if $n\ne0$.}
\end{cases}
\]
Similarly, applying $\cs(\farg,A)$ and $\cs(\farg,B)$ to \eqref{eq:tri1}, we obtain 
\[
\cs(D,\sh[n]A)=\begin{cases}
\K f'' & \text{if $n=1$,} \\
0 & \text{if $n\ne1$,}
\end{cases}
\qquad
\cs((D,\sh[n]B)=0.
\]
Moreover, applying $\cs(D,\farg)$ to \eqref{eq:tri1}, we obtain the exact sequence
\[
\cs(D,\sh[n]B)\to\cs(D,\sh[n]D)\mor{\sh[n]{f''}\comp}\cs(D,\sh[n+1]A)\to\cs(D,\sh[n+1]B).
\]
Since $\cs(D,\sh[n]B)=\cs(D,\sh[n+1]B)=0$, the map $\sh[n]{f''}\comp$ is an isomorphism. By what we have already proved, it follows that $D$ is exceptional. Summing up, so far we have proved that $\cs(X,\sh[n]Y)$ is as stated in \eqref{triangles2} when $X,Y\in\{A,B,D\}$. In a completely similar way, using \eqref{eq:tri2} (respectively \eqref{eq:tri3}), it can be shown that the same is true when $X,Y\in\{A,C,E\}$ (respectively $X,Y\in\{B,C,F\}$).

Applying $\cs(A,\farg)$ (respectively $\cs(\farg,C)$) to \eqref{eq:tri4} and using the already proved equalities $\cs(A,\sh[n]D)=\cs(A,\sh[n]E)=0$ (respectively $\cs(E,\sh[n]C)=\cs(F,\sh[n]C)=0$), we get $\cs(A,\sh[n]F)=0$ (respectively $\cs(D,\sh[n]C)=0$). On the other hand, applying $\cs(\farg,A)$ to \eqref{eq:tri3} (respectively $\cs(C,\farg)$ to \eqref{eq:tri1}) and using the already known equalities $\cs(B,\sh[n]A)=\cs(C,\sh[n]A)=\cs(C,\sh[n]B)=0$, we get $\cs(F,\sh[n]A)=0$ (respectively $\cs(C,\sh[n]D)=0$).

Applying $\cs(B,\farg)$ and $\cs(\farg,B)$ to \eqref{eq:tri2} and using the already known equalities $\cs(B,\sh[n]A)=\cs(C,\sh[n]B)=0$, we obtain
\[
\cs(B,\sh[n]E)=\begin{cases}
\K h'\comp g=\K l & \text{if $n=0$,} \\
0 & \text{if $n\ne0$,}
\end{cases}
\qquad
\cs(E,\sh[n]B)=\begin{cases}
\K\sh f\comp h''=\K l' & \text{if $n=1$,} \\
0 & \text{if $n\ne1$.}
\end{cases}
\]
To conclude the proof of the fact that $\cs(X,\sh[n]Y)$ is as stated in \eqref{triangles2} for every $X,Y\in\I$, it remains to show that this is true when $X,Y\in\{D,E,F\}$ (with $X\ne Y$). Now, applying $\cs(D,\farg)$ to \eqref{eq:tri2} and using the already proved equality $\cs(D,\sh[n]C)=0$, yields an exact sequence
\[
\cs(D,\sh[n]C)=0\to\cs(D,\sh[n]E)\mor{h''\comp}\cs(D,\sh[n+1]A)\to0=\cs(D,\sh[n+1]C).
\]
Thus $h''\comp$ is an isomorphism, which implies (since $f''=h''\comp k$)
\[
\cs(D,\sh[n]E)=\begin{cases}
\K k & \text{if $n=0$,} \\
0 & \text{if $n\ne0$.}
\end{cases}
\]
As we also know that $D$ and $E$ are exceptional, reasoning as in the first part of the proof we conclude that $\cs(X,\sh[n]Y)$ is as stated in \eqref{triangles2} when $X,Y\in\{D,E,F\}$. 

From the above description of the Hom spaces it is clear that, given $X,Y\in\I$ and $m,n\in\ZZ$, $\sh[n]X$ is exceptional (whence indecomposable) and $\sh[n]X\iso\sh[m]Y$ if and only if $n=m$ and $X=Y$. Moreover, each relation in \eqref{eq:comp} either holds by definition, or is due to the fact that the composition of two consecutive morphisms in a distinguished triangle is $0$. A direct check shows that every possible composition of morphisms between indecomposable objects of $\cs'$ is determined by these relations. Note that this proves \eqref{triangles1} and \eqref{triangles2}, once we show that $\cs'=\cs$. In any case $\cs'$ is a Krull-Schmidt category, since the same is true for $\cs$ by \autoref{KS}.

Passing to \eqref{triangles3}, we first observe that the octahedron is good (this follows immediately from \autoref{goodcrit}) and that $\Mor(\cs')$ is a Krull-Schmidt category. Indeed, as in the proof of \autoref{KS}, this is due to the fact that it is idempotent complete (which is a straightforward consequence of $\cs'$ being idempotent complete) and the endomorphism ring of every object has finite dimension over $\K$. Now we claim that a morphism of $\cs'$ is indecomposable in $\Mor(\cs')$ if and only if it is, up to isomorphism and shift, either one of the morphisms appearing in one of the triangles \eqref{eq:tri1}, \eqref{eq:tri2}, \eqref{eq:tri3}, \eqref{eq:tri4}, \eqref{eq:tri5}, \eqref{eq:tri6}, or $\id_X$ or $X\to0$ or $0\to X$ for some $X\in\I$. Assuming the claim, and recalling that a direct sum of distinguished triangles is distinguished and that two distinguished triangles extending the same morphism are isomorphic, we deduce that $\cs'$ is a triangulated subcategory of $\cs$ and that a triangle in $\cs'$ is distinguished if and only if it is of the form stated in \eqref{triangles3}. As $\cs$ is the triangulated envelope of $\ce$ and $\ce\subseteq\cs'$, it is then obvious that $\cs'=\cs$. Moreover, it should be clear that $\cs'$ is determined by $\ce$, up to isomorphism in $\Ht$. Since $H^0(\Pretr(\ce))$ is an algebraic triangulated category with a full and strong exceptional sequence as in \eqref{eq:S}, we conclude that $\cs\iso H^0(\Pretr(\ce))$ is algebraic.

In order to prove the claim, we will need the following easy result.

\begin{lem}\label{deccrit}
Consider a morphism of the form
\begin{equation}\label{eq:dec}
V\oplus V'\mor{u=\bigl(
\begin{smallmatrix}
\alpha & \beta \\ 
\gamma & \delta
\end{smallmatrix}
\bigr)
}W\oplus W'
\end{equation}
in an additive category $\ca$. If $\beta$ and $\gamma$ both factor through $\alpha$ in $\ca$, then $u\iso\alpha\oplus\tilde\delta$ in $\Mor(\ca)$, for some $\tilde\delta\in\ca(V',W')$.
\end{lem}

\begin{proof}
Suppose that $\beta=\alpha\comp\tilde\beta$, $\gamma=\tilde\gamma\comp\alpha$ for some $\tilde\beta\in\ca(V',V)$, $\tilde\gamma\in\ca(W,W')$. Setting $\tilde{\delta}:=\delta-\tilde{\gamma}\comp\alpha\comp\tilde{\beta}$, we have an isomorphism
\[
\xymatrix@C=3em{
V\oplus V'\ar[d]_-{\bigl(\begin{smallmatrix}
1 & \tilde{\beta}\\
0 & 1
\end{smallmatrix}\bigr)} 
\ar[r]^-{u} 
&W\oplus W'\ar[d]^-{\bigl(\begin{smallmatrix}
1 & 0\\
-\tilde{\gamma} & 1
\end{smallmatrix}\bigr)} \\
V\oplus V' \ar[r]^-{\bigl(\begin{smallmatrix}
\alpha & 0\\
0 & \tilde{\delta}
\end{smallmatrix}\bigr)}& W\oplus W'
}
\]
from $u$ to $\alpha\oplus\tilde\delta$ in $\Mor(\ca)$.
\end{proof}

Now we prove the claim. The other implication being very easy, we just need to prove that, if a morphism $u\colon Y\to Z$ of $\cs'$ is indecomposable in $\Mor(\cs')$, then it has the required form. By \eqref{triangles1} we can assume 
\[
Y=\bigoplus_{n\in\ZZ,\,X\in\I}\sh[n]X\otimes V_{X,n}, \qquad Z=\bigoplus_{n\in\ZZ,\,X\in\I}\sh[n]X\otimes W_{X,n},
\]
where $V_{X,n},W_{X,n}\in\fMod\K$ (and only a finite number of them are nonzero). Taking into account \eqref{triangles2}, the components of $u$ that can possibly be nonzero will be denoted as follows. For every $n\in\ZZ$ and every $X\in\I$ the component $\sh[n]X\otimes V_{X,n}\to\sh[n]X\otimes W_{X,n}$ is $\id_{\sh[n]X}\otimes i_{X,n}$ for some $i_{X,n}\in\fMod\K(V_{X,n},W_{X,n})$. For every $n\in\ZZ$ and every $x\colon X\to\sh[m]{X'}$ in $\J$ (where $m=1$ if $x\in\{f'',g'',h'',k'',l'\}$, and otherwise $m=0$) the component $\sh[n]X\otimes V_{X,n}\to\sh[n+m]{X'}\otimes W_{X',n+m}$ is $\sh[n]x\otimes x_n$ for some $x_n\in\fMod\K(V_{X,n},W_{X',n+m})$.

If there exist $n\in\ZZ$ and $X\in\I$ such that $i_{X,n}\ne0$, then we can find $v\in V_{X_n}$ such that $w:=i_{X,n}(v)\ne0$. Choosing decompositions $Y=V\oplus V'$ and $Z=W\oplus W'$ in $\cs'$ with $V=\sh[n]X\otimes\gen v$ and $W=\sh[n]X\otimes\gen w$, the morphism $u$ has the form \eqref{eq:dec} with $\alpha$ an isomorphism. Then clearly $\beta$ and $\gamma$ factor through $\alpha$, and so \autoref{deccrit} implies $u\iso\alpha\iso\id_{\sh[n] X}$ (since $u$ is indecomposable). Hence from now on we can assume 
\begin{equation}\label{eq:dec0}
i_{X,n}=0\text{ for every $n\in\ZZ$ and every $X\in\I$}.
\end{equation}

If there exists $n\in\ZZ$ such that $\ker(f'_n)\not\subseteq\ker(g_n)$, then we can find $v\in\ker(f'_n)$ such that $w:=g_n(v)\ne0$. Choosing decompositions $Y=V\oplus V'$ and $Z=W\oplus W'$ with $V=\sh[n]B\otimes\gen v$ and $W=\sh[n]C\otimes\gen w$, the morphism $u$ has the form \eqref{eq:dec} with $\gamma$ having $0$ component to $\sh[n]D\otimes W_{D,n}$. It follows from \eqref{triangles2} and \eqref{eq:dec0} that $\beta$ and $\gamma$ factor through $\alpha$, whence $u\iso\alpha\iso\sh[n]g$ by \autoref{deccrit}. Thus we can assume $\ker(f'_n)\subseteq\ker(g_n)$, and similarly $\ker(g_n)\subseteq\ker(f'_n)$. Now, if $\ker(f'_n)=\ker(g_n)\ne V_{B,n}$, then we can find $v\in V_{B,n}$ such that $w:=g_n(v)\ne0$ and $w':=f'_n(v)\ne0$. Choosing decompositions $Y=V\oplus V'$ and $Z=W\oplus W'$ with
\begin{gather*}
V=\sh[n]B\otimes\gen v, \qquad W=\sh[n]C\otimes\gen w\oplus\sh[n]D\otimes\gen{w'}, \qquad \sh[n]B\otimes\ker(g_n)\subseteq V', \\
(\sh[n]g\otimes g_n)\big(V'\cap(\sh[n]B\otimes V_{B,n})\big)\subseteq W'\cap(\sh[n]C\otimes W_{C,n}), \\
(\sh[n]{f'}\otimes f'_n)\big(V'\cap(\sh[n]B\otimes V_{B,n})\big)\subseteq W'\cap(\sh[n]D\otimes W_{D,n}),
\end{gather*}
the morphism $u$ has the form \eqref{eq:dec} with $\beta$ having $0$ component from $V'\cap(\sh[n]B\otimes V_{B,n})$. It follows from \eqref{triangles2} and \eqref{eq:dec0} that $\beta$ and $\gamma$ factor through $\alpha$, whence $u\iso\alpha\iso\sh[n]{\bigl(\begin{smallmatrix}
g \\
f'
\end{smallmatrix}\bigr)}$ by \autoref{deccrit}. So we can assume
\begin{equation}\label{eq:dec1}
f'_n=g_n=0\text{ for every $n\in\ZZ$,}
\end{equation}
and in a completely similar way also
\begin{equation}\label{eq:dec2}
h''_n=k'_n=0\text{ for every $n\in\ZZ$.}
\end{equation}

If there exists $n\in\ZZ$ such that $\im(k_n)\not\subseteq\im(h'_n)$, then we can find $v\in V_{D,n}$ such that $w:=k_n(v)\not\in\im(h'_n)$. Choosing decompositions $Y=V\oplus V'$ and $Z=W\oplus W'$ with $V=\sh[n]D\otimes\gen v$, $W=\sh[n]E\otimes\gen w$ and $\sh[n]E\otimes\im(h'_n)\subseteq W'$, the morphism $u$ has the form \eqref{eq:dec} with $\beta$ having $0$ component from $\sh[n]C\otimes V_{C,n}$. It follows from \eqref{triangles2} and \eqref{eq:dec0} that $\beta$ and $\gamma$ factor through $\alpha$, whence $u\iso\alpha\iso\sh[n]k$ by \autoref{deccrit}. Thus we can assume $\im(k_n)\subseteq\im(h'_n)$, and similarly $\im(h'_n)\subseteq\im(k_n)$. Now, if $\im(h'_n)=\im(k_n)\ne0$, then we can find $v\in V_{D,n}$ and $v'\in V_{C,n}$ such that $w:=k_n(v)=h'_n(v')\ne0$. Choosing decompositions $Y=V\oplus V'$ and $Z=W\oplus W'$ with $V=\sh[n]C\otimes\gen{v'}\oplus\sh[n]D\otimes\gen v$ and $W=\sh[n]E\otimes\gen w$, the morphism $u$ has the form \eqref{eq:dec} with $\gamma$ having $0$ component to $W'\cap(\sh[n]E\otimes V_{E,n})$. It follows from \eqref{triangles2} and \eqref{eq:dec0} that $\beta$ and $\gamma$ factor through $\alpha$, whence $u\iso\alpha\iso\sh[n]{\begin{pmatrix}
h' & -k
\end{pmatrix}}$ by \autoref{deccrit}. So we can assume
\begin{equation}\label{eq:dec3}
h'_n=k_n=0\text{ for every $n\in\ZZ$,}
\end{equation}
and in a completely similar way also
\begin{equation}\label{eq:dec4}
g''_n=f_n=0\text{ for every $n\in\ZZ$.}
\end{equation}

If $u\ne0$, then by \eqref{eq:dec0}, \eqref{eq:dec1}, \eqref{eq:dec2}, \eqref{eq:dec3} and \eqref{eq:dec4} there exist $n\in\ZZ$ and $x\in\{f'',g',h,k'',l,l'\}$ (denoted by $x\colon X\to\sh[m]{X'}$) such that $x_n\ne0$. Hence there exists $v\in V_{X,n}$ such that $w:=x_n(v)\ne0$. Choosing decompositions $Y=V\oplus V'$ and $Z=W\oplus W'$ with $V=\sh[n]X\otimes\gen v$ and $W=\sh[n+m]{X'}\otimes\gen w$, the morphism $u$ has the form \eqref{eq:dec} with $\beta$ and $\gamma$ factoring through $\alpha$. As usual, by \autoref{deccrit} this implies $u\iso\alpha\iso\sh[n]x$.

Finally, if $u=0$, then obviously $u$ is isomorphic to either $\sh[n]X\to0$ or $0\to\sh[n]X$ for some $n\in\ZZ$ and some $X\in\I$.

\section{Proof of \autoref{prop:oct}}\label{noniso}

In the following we replace $\Db(\fMod R)$ with the equivalent triangulated category $\ct:=\Kb(\proj R)$, namely the bounded homotopy category of finitely generated projective $R$-modules. Then, for every $n\in\NN$, the $R$-module $R_n$ is identified with the complex
\[
P_n:=(\cdots\to 0\to R\mor{x^n}R\to0\cdots)
\]
in degrees $-1$ and $0$. Given $m,n\in\NN$, we define the following morphisms in $\ct$:
\begin{itemize}
\item $x^i_{m,n}\colon P_m\to P_n$, where $i\ge0,n-m$, is represented by the morphism of complexes:
\begin{equation}\label{eq:x^i}
\xymatrix{
0 \ar[r] & R \ar[r]^-{x^m} \ar[d]_{x^{i+m-n}} & R \ar[d]^{x^i} \ar[r] & 0 \\
0 \ar[r] & R \ar[r]^-{x^n} & R \ar[r] & 0.
}
\end{equation}
Note that $x^i_{m,n}=0$ for $i\ge n$, since \eqref{eq:x^i} is null-homotopic in that case. Obviously $x^i_{m,n}$ can be identified with the morphism denoted by $x^i\colon R_m\to R_n$ before.
\item $y^i_{m,n}\colon P_m\to\sh{P_n}$, where $i\le m$, is represented by the morphism of complexes:
\begin{equation}\label{eq:y^i}
\xymatrix{
0 \ar[r] & 0 \ar[r] & R \ar[r]^-{x^m} \ar[d]^{x^{m-i}} & R \ar[r] & 0 \\
0 \ar[r] & R \ar[r]^-{-x^n} & R \ar[r] & 0 \ar[r] & 0.
}
\end{equation}
Note that $y^i_{m,n}=0$ for $i\le\max\{0,m-n\}$, since \eqref{eq:y^i} is null-homotopic in that case. It is also easy to see that $y^m_{m,n}$ can be identified with the morphism denoted by $y\colon R_m\to\sh{R_n}$ before.
\end{itemize}

We first prove some lemmas.

\begin{lem}\label{triRn}
Let $l$, $m$ and $n$ be positive integers.
\begin{enumerate}
\item\label{triRn1} $\{x^i_{m,n}\st0,n-m\le i<n\}$ is a basis of $\ct(P_m,P_n)$. Given $0,m-l\le i<m$ and $0,n-m\le j<n$, we have
\[
x^j_{m,n}\comp x^i_{l,m}=\begin{cases}
x^{i+j}_{l,n} & \text{if $i+j<n$} \\
0 & \text{otherwise}.
\end{cases}
\]
\item\label{triRn2} $\{y^i_{m,n}\st0,m-n<i\le m\}$ is a basis of $\ct(P_m,\sh{P_n})$. Given $0,m-l\le i<m$ and $0,m-n<j\le m$, we have
\[
y^j_{m,n}\comp x^i_{l,m}=\begin{cases}
y^{j-i}_{l,n} & \text{if $j-i>0,l-n$} \\
0 & \text{otherwise}.
\end{cases}
\]
Given $0,l-m<i\le l$ and $0,n-m\le j<n$, we have
\[
\sh{x^j_{m,n}}\comp y^i_{l,m}=\begin{cases}
y^{i-j}_{l,n} & \text{if $i-j>0,l-n$} \\
0 & \text{otherwise}.
\end{cases}
\]
\item\label{triRn3} Given $0,n-m\le i<n$, the triangle
\begin{equation}\label{eq:triangle}
\tri{P_m}{x^i_{m,n}}{P_n}{\Bigl(\begin{smallmatrix}
y^n_{n,m-n+i} \\
x^0_{n,i}
\end{smallmatrix}\Bigr)}{\sh{P_{m-n+i}}\oplus P_i}{(\begin{smallmatrix}
-\sh{x^{n-i}_{m-n+i,m}} &\quad y^i_{i,m}
\end{smallmatrix})}
\end{equation}
is distinguished in $\ct$.
\item\label{triRn4} The triangles
\begin{gather}
\label{eq:good1}
\tri{P_3}{\Bigl(\begin{smallmatrix}
y^2_{3,2} \\
x^1_{3,2}
\end{smallmatrix}\Bigr)}{\sh{P_2}\oplus P_2}{\Biggl(
\begin{smallmatrix}
\sh{x^0_{2,1}} & 0 \\
0 & x^0_{2,1} \\
\sh{x^1_{2,3}} & - y^2_{2,3}
\end{smallmatrix}\Biggr)
}{\sh{P_1}\oplus P_1\oplus\sh{P_3}}{(
\begin{smallmatrix}
-\sh{x^2_{1,3}} &\quad y^1_{1,3} &\quad \sh{x^1_{3,3}}   
\end{smallmatrix})
}, \\
\label{eq:good2}
\tri {P_3}{\Biggl(
\begin{smallmatrix}
x^1_{3,3} \\
y^3_{3,1} \\
x^0_{3,1}
\end{smallmatrix}\Biggr)
}{P_3\oplus \sh{P_1}\oplus P_1}{\Bigl(
\begin{smallmatrix}
y^3_{3,2} & -\sh{x^1_{1,2}} & 0 \\
x^0_{3,2} & 0 & -x^1_{1,2}
\end{smallmatrix}\Bigr)
}{\sh{P_2}\oplus P_2}{(\begin{smallmatrix}
\sh{-x^2_{2,3}} &\quad y^1_{2,3}
\end{smallmatrix})}
\end{gather}
are distinguished in $\ct$.
\end{enumerate}
\end{lem}

\begin{proof}
Parts \eqref{triRn1} and \eqref{triRn2} are easy and probably known, so we just sketch the proof of \eqref{triRn3} and \eqref{triRn4}. 

As for \eqref{triRn3}, there is a distinguished triangle
\begin{equation}\label{eq:triangle1}
\xymatrix{
P_m \ar[r]^-{x^i_{m,n}} & P_n \ar[r]^-j & \cone{x^i_{m,n}} \ar[r]^-q & \sh{P_m}
}
\end{equation}
in $\ct$, where $j$ and $q$ are given by the following morphisms of complexes:
\begin{equation*}
\xymatrix{
0\ar[r]&  0\ar[d] \ar[rr] && R\ar[d]^-{\bigl(\begin{smallmatrix}
    0\\
    1
\end{smallmatrix}\bigr)} \ar[rr]^{x^n} && R \ar@{=}[d]\ar[r]& 0 \\
0\ar[r]& R\ar@{=}[d]\ar[rr]^-{\bigl(\begin{smallmatrix}
    -x^m\\
    x^{i+m-n}
\end{smallmatrix}\bigr)} && R\oplus R\ar[d]^-{( \begin{smallmatrix}
    1 & 0
\end{smallmatrix})} \ar[rr]^-{(\begin{smallmatrix}
    x^i& x^n
\end{smallmatrix})} && R\ar[d]\ar[r]& 0\\
0\ar[r]&  R \ar[rr]^-{-x^m} && R\ar[rr] && 0 \ar[r]& 0. 
}
\end{equation*}
Since $\cone{x^i_{m,n}}$ is isomorphic to $\sh{P_{m-n+i}}\oplus P_i$ via the isomorphism
\[
\xymatrix{
0\ar[r]& R\ar@<-0.5ex>[d]_-{-1} \ar[rr]^-{\bigl(\begin{smallmatrix}
    -x^m\\
    x^{i+m-n}
\end{smallmatrix}\bigr)} && R\oplus R \ar@<0.5ex>[d]^-{\bigl(\begin{smallmatrix}
    0 & 1 \\
    1 & x^{n-i}
\end{smallmatrix}\bigr)} \ar[rr]^-{(\begin{smallmatrix}
    x^i& x^n
\end{smallmatrix})}
&& R\ar[r]\ar@{=}[d]&0\\
0\ar[r]& R \ar@<-0.5ex>[u]_-{-1} \ar[rr]_-{\bigl(\begin{smallmatrix}
-x^{i+m-n} \\
0
\end{smallmatrix}\bigr)} && R\oplus R \ar@<0.5ex>[u]^-{\bigl(\begin{smallmatrix}
    -x^{n-i} & 1\\
    1 & 0
\end{smallmatrix}\bigr)} \ar[rr]_-{(\begin{smallmatrix}
0 & x^i
\end{smallmatrix})}
&& R\ar[r]&0,
}
\]
it is easy to conclude that the triangle \eqref{eq:triangle} is isomorphic to \eqref{eq:triangle1}, hence it is distinguished, as well.

Coming to \eqref{triRn4}, we consider only \eqref{eq:good1}, as \eqref{eq:good2} is similar. Setting $u:=\Bigl(\begin{smallmatrix}
y^2_{3,2} \\
x^1_{3,2}
\end{smallmatrix}\Bigr)$, there is a distinguished triangle
\begin{equation}\label{eq:triangle2}
\xymatrix{
P_3 \ar[r]^-u & {\sh{P_2}\oplus P_2} \ar[r]^-{j}& \cone u \ar[r]^-{q}& \sh{P_3}
}
\end{equation}
in $\ct$, where $j$ and $q$ are given by the following morphisms of complexes:
\[
\xymatrix{
0\ar[r]& R\ar[d]^-{\bigl( 
\begin{smallmatrix}
0\\
1
\end{smallmatrix}
\bigr)} \ar[rr]^-{\bigl( 
\begin{smallmatrix}
-x^2\\
0
\end{smallmatrix}
\bigr)} && R\oplus R \ar[rr]^{( 
\begin{smallmatrix}
0 & x^2
\end{smallmatrix}
)} \ar[d]^-{\Bigl( 
\begin{smallmatrix}
0 & 0\\
1 & 0\\
0 & 1
\end{smallmatrix}
\Bigr)}\ar[rr]&& R\ar@{=}[d] \ar[r]& 0 \\
0\ar[r]& R\oplus R \ar[rr]^-{\biggl(\begin{smallmatrix}
-x^3&0\\
x& -x^2\\
x^2&0
\end{smallmatrix}\biggr)}\ar[d]^-{(\begin{smallmatrix}
1 & 0
\end{smallmatrix}
)}  && R\oplus R\oplus R\ar[d]^-{( \begin{smallmatrix}
1 & 0 &0
\end{smallmatrix}
)} \ar[rr]^-{(\begin{smallmatrix}
x& 0& x^2
\end{smallmatrix})} 
&& R\ar[r]\ar[d]&0\\
0\ar[r]& R \ar[rr]^{-x^3} && R \ar[rr]&& 0 \ar[r]& 0.
}
\]
Since $\cone u$ is isomorphic to $\sh{P_1}\oplus P_1\oplus\sh{P_3}$ via the isomorphism
\[
\xymatrix{
0 \ar[r] & R\oplus R \ar@<0.5ex>[d]^{\bigl(\begin{smallmatrix}
    -1 & x \\
    0 &1
\end{smallmatrix}\bigr)} \ar[rr]^-{\biggl(\begin{smallmatrix}
-x^3&0\\
x& -x^2\\
x^2&0
\end{smallmatrix}\biggr)}  && R\oplus R\oplus R \ar@<0.5ex>[d]^{\Bigl(\begin{smallmatrix}
0 & 1 & 0\\
1 & 0 & x\\
0 & x &-1
\end{smallmatrix}\Bigr)} \ar[rr]^-{(\begin{smallmatrix}
x & 0 & x^2
\end{smallmatrix})} 
&& R\ar[r]\ar@{=}[d] & 0 \\
0 \ar[r] & R\oplus R\ar@<0.5ex>[u]^{\bigl(\begin{smallmatrix}
    -1 & x \\
    0 &1
\end{smallmatrix}\bigr)} \ar[rr]_-{\Bigl(\begin{smallmatrix}
-x & 0 \\
0 & 0 \\
0 & -x^3
\end{smallmatrix}\Bigr)} && R\oplus R\oplus R \ar@<0.5ex>[u]^{\Bigl(\begin{smallmatrix}
-x^2 & 1 & x \\
1 & 0 & 0 \\
x & 0 & -1
\end{smallmatrix}\Bigr)} \ar[rr]_-{(\begin{smallmatrix}
0 & x & 0
\end{smallmatrix})} 
&& R\ar[r]&0,
}
\]
it is easy to conclude that the triangle \eqref{eq:good1} is isomorphic to \eqref{eq:triangle2}, hence it is distinguished, as well.
\end{proof}

\begin{lem}\label{gooddefor}
Assume that \eqref{eq:oct} is an octahedron and that there exists $\epsilon\colon E\to E$ such that $\epsilon^2=0$ and
\begin{equation}
\label{eq:epsilon}
\epsilon\comp l=0, \qquad l'\comp\epsilon=0, \qquad h''\comp\epsilon\comp k=0, \qquad k'\comp\epsilon\comp h'=0.
\end{equation}
Then \eqref{eq:oct} remains an octahedron if we replace $k$ and $k'$ with
\begin{gather*}
\tilde{k}:=(1+\epsilon)\comp k, \qquad \tilde{k}':=k'\comp (1-\epsilon)
\end{gather*}
(keeping all the other maps unchanged).
Moreover, this new octahedron is good if the original one is good and there exist two nilpotent morphisms $\epsilon',\epsilon'':E\to E$ such that:
\begin{gather}
\label{eq:epsilon'}
\epsilon'\comp l=0, \qquad h''\comp\epsilon'=0, \qquad k'\comp(\epsilon'-\epsilon)=0, \\
\label{eq:epsilon''}
l'\comp \epsilon''=0, \qquad \epsilon''\comp h'=0, \qquad (\epsilon''-\epsilon)\comp k=0.
\end{gather}
\end{lem}

\begin{proof}
Since $\epsilon^2=0$, $1+\epsilon$ is an automorphism  of $E$ with inverse $1-\epsilon$. It follows that the triangle $\tri D {\tilde{k}} E {\tilde{k}'} F {k''}$ is isomorphic to the distinguished triangle $\tri D {k} E {k'} F {k''}$, and so it is distinguished. The commutativity of all the squares involved follows from \eqref{eq:epsilon}.

It remains to prove that the new octahedron is good under the extra assumptions. First observe that, since $\epsilon'$ and $\epsilon''$ are nilpotent, $1-\epsilon'$ and $1+\epsilon''$ are automorphisms of $E$.
By \eqref{eq:epsilon'} there is an isomorphism of triangles
\begin{equation*}
\xymatrix{
B\ar[r]^l\ar@{=}[d]&E \ar[d]_{1-\epsilon'} \ar[r]^-{\bigl(\begin{smallmatrix}
\tilde{k}'  \\
-h''
\end{smallmatrix}\bigr)} & F\oplus \sh{A}\ar@{=}[d]\ar[rr]^-{(\begin{smallmatrix}
g'' & \sh f
\end{smallmatrix})}&& \sh{B}\ar@{=}[d]
\\
B\ar[r]^l&E \ar[r]^-{\bigl(\begin{smallmatrix}
{k}'  \\
-h''
\end{smallmatrix}\bigr)} & F\oplus \sh{A}\ar[rr]^-{(\begin{smallmatrix}
g'' & \sh f
\end{smallmatrix})}&& \sh{B}.
}
\end{equation*}
In the same vein, by \eqref{eq:epsilon''} there is an isomorphism of triangles
\[
\xymatrix{
B\ar@{=}[d]\ar[r]^-{\bigl(\begin{smallmatrix}
g  \\
f'
\end{smallmatrix}\bigr)}&C\oplus D\ar@{=}[d] \ar[rr]^-{(\begin{smallmatrix}
h' & -{k}
\end{smallmatrix})} && E\ar[r]^-{l'}\ar[d]^{1+\epsilon''}& \sh{B}\ar@{=}[d]\\
B\ar[r]^-{\bigl({\begin{smallmatrix}
g  \\
f'
\end{smallmatrix}\bigr)}}&C\oplus D \ar[rr]^-{(\begin{smallmatrix}
h' & -\tilde{k}
\end{smallmatrix})} && E\ar[r]^-{l'}& \sh{B}
}
\]
and we are done.
\end{proof}

We are now ready to prove \autoref{prop:oct}. First observe that, with the notation of this section, we have
\begin{gather*}
A=B=C=P_3, \qquad F=D=\sh{P_1}\oplus P_1, \qquad E=\sh{P_2}\oplus P_2,\\
f=g=x^1_{3,3}, \qquad f'=\Bigl(\begin{smallmatrix}
y^3_{3,1} \\
x^0_{3,1}
\end{smallmatrix}\Bigr),
\qquad h'=\Bigl(\begin{smallmatrix}
y^3_{3,2} \\
x^0_{3,2}
\end{smallmatrix}\Bigr),
\qquad g''=(\begin{smallmatrix}
-\sh{x^2_{1,3}} &\quad y^1_{1,3}
\end{smallmatrix}), \qquad h''=(\begin{smallmatrix}
-\sh{x^1_{2,3}} &\quad y^2_{2,3}
\end{smallmatrix}), \\
k=\Bigl(\begin{smallmatrix}
\sh{x^1_{1,2}} & 0\\
0 & x^1_{1,2}
\end{smallmatrix}\Bigr), \qquad k'=\Bigl(\begin{smallmatrix}
\sh{x^0_{2,1}} & 0\\
0 & x^0_{2,1}
\end{smallmatrix}\Bigr), \qquad \tilde{k}=\Bigl(\begin{smallmatrix}
\sh{x^1_{1,2}} & y^1_{1,2} \\
0 & x^1_{1,2}
\end{smallmatrix}\Bigr), \qquad \tilde{k}'=\Bigl(\begin{smallmatrix}
\sh{x^0_{2,1}} & -y^2_{2,1} \\
0 & x^0_{2,1}
\end{smallmatrix}\Bigr).
\end{gather*}
First we prove that the morphisms in \eqref{eq:oct3} and in \eqref{eq:oct3b} give rise to a good octahedron. Using \eqref{triRn1} and \eqref{triRn2} of \autoref{triRn} it is immediate to see that
\begin{gather*}
f''=(\begin{smallmatrix}
\sh{-x^2_{1,3}} &\quad y^1_{1,3}
\end{smallmatrix}), \qquad h=x^2_{3,3}, \qquad g'=\Bigl(\begin{smallmatrix}
y^3_{3,1} \\
x^0_{3,1}
\end{smallmatrix}\Bigr), \qquad k''=\Bigl(\begin{smallmatrix}
\sh{-y^1_{1,1}} & 0 \\
0 & y^1_{1,1}
\end{smallmatrix}\Bigr), \\
l=\Bigl(\begin{smallmatrix}
y^2_{3,2} \\
x^1_{3,2}
\end{smallmatrix}\Bigr), \qquad l'=(\begin{smallmatrix}
\sh{-x^2_{2,3}} &\quad y^1_{2,3}
\end{smallmatrix}).
\end{gather*}
Hence the triangles \eqref{eq:tri1}, \eqref{eq:tri2}, \eqref{eq:tri3} and \eqref{eq:tri4} become
\begin{gather*}
\tri{P_3}{x^1_{3,3}}{P_3}{\Bigl(\begin{smallmatrix}
y^3_{3,1} \\
x^0_{3,1}
\end{smallmatrix}\Bigr)}{\sh{P_1}\oplus P_1}{(\begin{smallmatrix}
\sh{-x^2_{1,3}} &\quad y^1_{1,3}
\end{smallmatrix})}, \\
\tri{P_3}{x^2_{3,3}}{P_3}{\Bigl(\begin{smallmatrix}
y^3_{3,2} \\
x^0_{3,2}
\end{smallmatrix}\Bigr)}{\sh{P_2}\oplus P_2}{(\begin{smallmatrix}
\sh{-x^1_{2,3}} &\quad y^2_{2,3}
\end{smallmatrix})}, \\
\tri{P_3}{x^1_{3,3}}{P_3}{\Bigl(\begin{smallmatrix}
y^3_{3,1} \\
x^0_{3,1}
\end{smallmatrix}\Bigr)}{\sh{P_1}\oplus P_1}{(\begin{smallmatrix}
\sh{-x^2_{1,3}} &\quad y^1_{1,3}
\end{smallmatrix})}, \\
\tri{\sh{P_1}\oplus P_1}{\Bigl(\begin{smallmatrix}
\sh{x^1_{1,2}} & 0 \\
0 & x^1_{1,2}
\end{smallmatrix}\Bigr)}{\sh{P_2}\oplus P_2}{\Bigl(\begin{smallmatrix}
\sh{x^0_{2,1}} & 0 \\
0 & x^0_{2,1}
\end{smallmatrix}\Bigr)}{\sh{P_{1}}\oplus P_1}{\Bigl(\begin{smallmatrix}
\sh{-y^1_{1,1}} & 0 \\
0 & y^1_{1,1}
\end{smallmatrix}\Bigr)}.
\end{gather*}
As they are all distinguished by \eqref{triRn3}, we really have an octahedron. Moreover, \eqref{eq:tri5} and \eqref{eq:tri6} become precisely \eqref{eq:good1} and \eqref{eq:good2}. Hence \eqref{triRn4} directly implies that the octahedron is good.

The fact that the morphisms \eqref{eq:oct3} and \eqref{eq:octdata2} form a good octahedron follows easily from \autoref{gooddefor}, taking
\begin{equation}
\epsilon:={\Bigl(\begin{smallmatrix}
0 & y^2_{2,2}\\
0 & 0
\end{smallmatrix}\Bigr)},
\qquad
\epsilon':={\Bigl(\begin{smallmatrix}
0 & y^2_{2,2}\\
0 & x^1_{2,2}
\end{smallmatrix}\Bigr)},
\qquad
\epsilon'':={\Bigl(\begin{smallmatrix}
\sh{-x^1_{2,2}} & y^2_{2,2}\\
0 & 0
\end{smallmatrix}\Bigr)}.
\end{equation}
Finally, we suppose there exists an isomorphism $\phi$ between the two octahedra. Without loss of generality, such an isomorphism has the form
\begin{gather*}
\phi_A=\alpha_A x^0_{3,3}+\beta_A x^1_{3,3}+\gamma_A x^2_{3,3},\qquad 
\phi_B=\alpha_B x^0_{3,3}+\beta_B x^1_{3,3}+\gamma_B x^2_{3,3},\\ 
\phi_C=\alpha_C x^0_{3,3}+\beta_C x^1_{3,3}+\gamma_C x^2_{3,3},\qquad
\phi_D={\Bigl(\begin{smallmatrix}
    \alpha'_D \sh{x^0_{1,1}} & \lambda y^1_{1,1} \\
    0 & \alpha_D x^0_{1,1}
\end{smallmatrix}\Bigr)},\\
\phi_E={\Bigl(\begin{smallmatrix}
    \alpha'_E \sh{x^0_{2,2}}+\beta'_E \sh{x^1_{2,2}} & \mu y^2_{2,2}+\nu y^1_{2,2} \\
    0 & \alpha_E x^0_{2,2}+\beta_E x^1_{2,2}
\end{smallmatrix}\Bigr)}, \qquad
\phi_F={\Bigl(\begin{smallmatrix}
    \alpha'_F \sh{x^0_{1,1}} & \xi y^1_{1,1} \\
    0 & \alpha_F x^0_{1,1}
\end{smallmatrix}\Bigr)},
\end{gather*}
where $\phi_X\colon X\to X$ for every $X\in\I$, $\alpha_X,\alpha'_X\in\K^{*}$ and $\beta_X,\beta'_X,\gamma_X,\lambda,\mu,\nu,\xi\in \K$. Up to multiplying everything by $\alpha_A^{-1}$, we can assume $\alpha_A=1$. As $f\comp\phi_A=\phi_B\comp f$ and $g\comp\phi_B=\phi_C\comp g$ we must also have $\alpha_B=\alpha_C=1$ and $\beta_A=\beta_B=\beta_C:=\beta$. The conditions $\phi_D\comp f'=f'\comp \phi_B$ and $g''\comp \phi_F
=\sh{\phi_B}\comp g''$ imply $\alpha_D=\alpha'_D=\alpha_F=\alpha'_F=1$.
Since $\tilde{k}'\comp \phi_E=\phi_F\comp k'$ we have $\alpha_E=\alpha'_E=\mu=1$. From $h''\comp\phi_E=\sh{\phi_A}\comp h''$ we get $\beta'_E=\beta$ and $\beta_E=\beta+1$. But this is not possible since 
\begin{equation*}
\phi_E\comp h'=
{\Bigl(\begin{smallmatrix}
y^3_{3,2}+(\beta+1)y^2_{3,2} \\ 
x^0_{3,2}+(\beta+1)x^1_{3,2}
\end{smallmatrix}\Bigr)}
\ne
{\Bigl(\begin{smallmatrix}
y^3_{3,2}+\beta y^2_{3,2} \\ 
x^0_{3,2}+\beta x^1_{3,2}
\end{smallmatrix}\Bigr)
}
=
h'\comp\phi_C.
\end{equation*}
It means there is no isomorphism between the two octahedra, and we are done.

\end{document}